\documentclass[11pt,english]{article}

\usepackage{amsmath,amssymb,amsthm,mathtools}
\usepackage{enumerate}

\usepackage{hyperref}
\hypersetup{colorlinks=true,linkcolor=blue,citecolor=blue,pdfpagemode=UseNone,pdfstartview={XYZ null null 1.00}}

\usepackage{amsmath, amsthm, amsopn, amssymb}
\usepackage{xcolor}
\usepackage{graphicx}

\usepackage[margin = 20mm, bottom=20mm,footskip=7mm,top=20mm]{geometry}
\theoremstyle{plain}
\newtheorem*{theorem*}{Theorem}
\newtheorem{theorem}{Theorem}[section]
\newtheorem{lemma}[theorem]{Lemma}
\newtheorem{claim}[theorem]{Claim}
\newtheorem{proposition}[theorem]{Proposition}
\newtheorem*{claim*}{Claim}

\theoremstyle{remark}

\newcommand{\A}{\mathcal{A}}

\newcommand{\F}{\mathcal{F}}
\newcommand{\G}{\mathcal{G}}

\let\originalleft\left
\let\originalright\right
\renewcommand{\left}{\mathopen{}\mathclose\bgroup\originalleft}
\renewcommand{\right}{\aftergroup\egroup\originalright}
\date{}

\title{\vspace{-0.8cm}Infinite Sperner's theorem}

\author{
Benny Sudakov\thanks{ETH Zurich, \emph{e-mail}: \textbf{\{benjamin.sudakov,istvan.tomon,zsolt.wagner\}@math.ethz.ch}. Research was supported by SNSF grant 200021-175573.}
\and
Istv\'an Tomon\footnotemark[1] \thanks{MIPT Moscow, Research partially supported by  the Ministry of Educational and Science of the Russian Federation in the framework of MegaGrant no 075-15-2019-1926.}
\and
Adam Zsolt Wagner\footnotemark[1]
}

\begin{document}
\maketitle

\begin{abstract}
One of the most classical results in extremal set theory is Sperner's theorem, which says that the largest antichain in the Boolean lattice $2^{[n]}$ has size $\Theta\big(\frac{2^n}{\sqrt{n}}\big)$. Motivated by an 
old problem of Erd\H{o}s on the growth of infinite Sidon sequences, in this note we study the growth rate of  maximum infinite antichains. Using the well known Kraft's inequality for prefix codes, it is not difficult to show that infinite antichains should be ``thinner'' than the corresponding finite ones. More precisely, if $\F\subset 2^{\mathbb{N}}$ is an antichain, then
   $$\liminf_{n\rightarrow \infty}\big|\F \cap 2^{[n]}\big|\left(\frac{2^n}{n\log n}\right)^{-1}=0.$$ 
Our main result shows that this bound is essentially tight, that is, we construct an antichain $\F$ such that $$\liminf_{n\rightarrow \infty}\big|\F \cap 2^{[n]}\big|\left(\frac{2^n}{n\log^{C} n}\right)^{-1}>0$$ holds for some absolute constant $C>0$.
\end{abstract}

\section{Introduction}
A typical question in extremal combinatorics asks to determine or estimate the maximum or minimum possible size of a combinatorial structure which satisfies certain requirements. 
Once such a result is established one can also ask about the growth rate of the infinite structure that satisfies the same requirements and compare its behavior with the one in the finite setting. Probably the first problem of this kind was posed by Erd\H{o}s in the 1940s. A \emph{Sidon set} is a set of natural numbers not containing any non-trivial solutions to the equation $a+b=c+d$. Denote by
$[n]$ the set $\{1, \dots, n\}$. It is known (see~\cite{sidon}) that for all large $n$ there exist Sidon sets $S \subset [n]$ of size at least $|S|\geq (1-o(1))\sqrt{n}$. On the other hand, it was already observed in 1941 by Erd\H{o}s and Tur\'an~\cite{erdosturan} that no infinite Sidon set $S \subset \mathbb{N}$  can achieve $\big|S \cap[n]\big| = \Theta(\sqrt{n})$ simultaneously for all $n$. This was further refined by Erd\H{o}s who showed~(see \cite{erdosupper}, Chapter II, §3) that if $S$ is a Sidon set, then $$\liminf_{n\rightarrow \infty}\frac{\big|S\cap [n]\big| }{\sqrt{n/\log n}}\leq 1 \,.$$ 
In particular, this shows that infinite Sidon sequences are thinner than the densest finite ones. The proof of this result appeared in a 1953 letter from Erd\H{o}s to St\"ohr. This letter is cited in~\cite{stohr}, which studies a number of other problems in additive combinatorics in the same spirit of ``finite versus infinite behavior''. A remarkable construction of Ruzsa~\cite{ruzsa} shows that there are Sidon sets $S$ with $\big|S\cap [n]\big|\geq n^{\sqrt{2}-1+o(1)}$ for all $n$. Closing this gap is a fascinating open problem. For further generalizations of this problem, we refer the reader to~\cite{rodl1,rodl2}.
 
It is only natural to study this phenomenon of ``finite versus infinite'' for other extremal problems as well. It is well known that Sidon sets and $C_4$-free graphs are intimately related. Hence, the corresponding question in the graph theoretic setting is that given an infinite $C_4$-free graph $G$ on vertex set $\mathbb{N}$, how large can the minimum degree $\delta_n$ of $G_n$ be, where $G_n$ is the restriction of $G$ to the set $[n]$? Conlon--Fox--Sudakov~\cite{benny!} proved the graph-theoretic analogue of Erd\H{o}s' result, that is, $$\liminf_{n\rightarrow\infty}\frac{\delta_n}{\sqrt{n/\log n}}<\infty.$$
They also extended this result to the more general setting of $K_{s,t}$-free graphs.

In this short note we consider a similar problem for maximal antichains. A family $\F\subset 2^{\mathbb{N}}$ is an \emph{antichain} if for all distinct $A,B\in \F$ we have $A\not\subset B$. One of the most classical results of extremal combinatorics is Sperner's theorem \cite{sperner}, which states that if $\F \subset 2^{[n]}$ is an antichain then $$|\F|\leq \binom{n}{\lfloor n/2 \rfloor} = \Theta \left(\frac{2^n}{\sqrt{n}}\right).$$ 
Following Erd\H{o}s, one can naturally ask what happens with Sperner's problem in an infinite setting? Using the well known Kraft's inequality \cite{kraft} from 1949, it is not difficult to show that there is a polynomial drop in density in this case.
\begin{theorem}\label{thm:upper} 
Let $\F\subset 2^{\mathbb{N}}$ be an antichain. Then  $$\liminf_{n\rightarrow \infty}\big|\F \cap 2^{[n]}\big| \, \left(\frac{2^n}{ n\log n}\right)^{-1}=0.$$
\end{theorem}
While in the case of Sidon sets there is a wide gap between the lower and upper bound, perhaps surprisingly, in our problem we can provide bounds that match up to a polylogarithmic term.
\begin{theorem}\label{thm:constr}
There exists an antichain $\F\subset 2^{\mathbb{N}}$ such that $$\liminf_{n\rightarrow \infty}\big|\F \cap 2^{[n]}\big|\, \left( \frac{2^n}{n\log^{46} n} \right)^{-1}>0.$$ 
\end{theorem}
We give the proof of the upper bound, Theorem~\ref{thm:upper}, in Section~\ref{sec:upper}. The construction that shows the lower bound in Theorem~\ref{thm:constr} is presented in Section~\ref{sec:constr}. Some open questions and future research directions are discussed in Section~\ref{sec:concl}.

\section{Infinite antichains cannot be too large}\label{sec:upper}
Consider a family $\mathcal{S}$ of $\{0,1\}$ sequences of finite length. Say that $\mathcal{S}$ is a \emph{prefix code} if no element of $\mathcal{S}$ is a prefix of another. Denoting by $|s|$ the length of a $\{0,1\}$ sequence, the well known Krafts's inequality \cite{kraft} tells us that if $\mathcal{S}$ is a prefix code, then
$$\sum_{s\in\mathcal{S}}2^{-|s|}\leq 1.$$

We show that we can use this inequality to prove the following proposition, which then implies Theorem~\ref{thm:upper}.

\begin{proposition}\label{thm:sum}
Let $\F\subset 2^{\mathbb{N}}$ be an antichain. Then
$$\sum_{n=1}^{\infty}\frac{\big|\F\cap 2^{[n]}\big|}{2^{n}}\leq 2.$$
\end{proposition}

\begin{proof}
It is enough to show that for every positive integer $N$, we have 
$$\sum_{n=1}^{N}\frac{\big|\F\cap 2^{[n]}\big|}{2^{n}}\leq 2.$$
For a set $F \in \F$, let $\max(F)$ denote the maximum element of $F$, and for $1\leq n\leq N$, let $\A_n=\{F\in\F:\max(F)=n\}$. We can identify each set $A\in\mathcal{A}_{n}$ with a $\{0,1\}$ sequence $s_{A}$ of length $n$ such that the $i$-th element of $s_{A}$ is $0$ if $i\not\in A$, and $1$ otherwise. Then the fact that $\mathcal{F}$ is an antichain implies that the family $\{s_{A}:A\in \mathcal{F}\cap 2^{[N]}\}$ is a prefix code. Therefore, by the above mentioned Kraft's inequality, we get
$$\sum_{n=1}^{N}\frac{|\A_{n}|}{2^{n}}\leq 1.$$
Note that $|\A_n|=\big|\F\cap 2^{[n]}|-|\F\cap  2^{[n-1]}\big|$, so 
 $$\sum_{n=1}^{N}\frac{\big|\F\cap 2^{[n]}|-|\F\cap 2^{[n-1]}\big|}{2^{n}}\leq 1.$$ 
This can be rewritten as
$$\frac{\big|\mathcal{F}\cap 2^{[N]}\big|}{2^{N}}+\sum_{n=1}^{N-1}\frac{\big|\F\cap 2^{[n]}\big|}{2^{n+1}}\leq 1,$$
which implies the desired inequality.
\end{proof}

\begin{proof}[Proof of Theorem~\ref{thm:upper}]
For $n=1,2,\dots$, let $f_n=\big|\F\cap 2^{[n]}\big|\,\left(\frac{2^n}{n\log n}\right)^{-1}$. Then by Proposition~\ref{thm:sum}, we have
$$\sum_{n=1}^{\infty}\frac{f_n}{n\log n}\leq 2.$$
As the sum $\sum_{n=1}^{\infty}\frac{1}{n\log n}$ diverges, we conclude that $\liminf_{n\rightarrow\infty}f_n=0$.
\end{proof}

\section{Constructing dense infinite antichains}\label{sec:constr}

There are many easy constructions that match the upper bound of Kraft's inequality. Note that we cannot use these for our problem, because being an antichain is a much stronger requirement than being prefix-free, e.g.~the family $\{\{1,2,3\},\{1,3\}\}$ is prefix-free but not an antichain.

Before we give the construction that achieves the bound in Theorem~\ref{thm:constr}, it will be helpful to first look at a slightly worse construction. While this only achieves a density of $2^n/n^{3/2}$, it will illustrate some of the ideas of the proof of Theorem~\ref{thm:constr}. Consider the family
\begin{equation}\label{eq:constr1}
    \F_n=\{A\subset [2n]: |A|=n, ~ \forall i: 0<i<n\implies  |A\cap [2i]|<i \},
\end{equation}
and let $\F=\bigcup_n \F_n$. Observe that $\F$ is an antichain, and that there is a bijection between elements of $\F_n$ and lattice paths from $(0,0)$ to $(n,n)$ where every step is to the right or up, and the path is fully under the line $y=x$ except at the endpoints. The enumeration of such paths is given by the Catalan numbers, and hence $$|\F_n|=\frac{1}{n}\binom{2(n-1)}{n-1}=\Theta\left(\frac{2^{2n}}{n^{3/2}}\right).$$ One can then verify that $$|\F\cap 2^{[n]}|=\Theta\left(\frac{2^{n}}{n^{3/2}}\right).$$

We will generalize the previous construction as follows. Fix an arbitrary monotone increasing function $f:\mathbb{N}\rightarrow \mathbb{N}$, and define the family 
$$\F_n=\{A\subset [2n]: |A|=n+f(n),\forall i:0<i<n\implies |A\cap [2i]|<i+f(i)\}.$$
Again one can show that $\F=\bigcup_{n=1}^{\infty}\F_n$ is an antichain. Indeed, suppose to the contrary that there is $A\in \F_n$ and $B\in \F_m$ for some $n,m$ such that $A\subset B$. As each $\F_n$ is a uniform family, we must have $n\neq m$. Moreover as $f$ is a monotone increasing function, we must have $n<m$. Then $A$ contains $n+f(n)$ elements from $[2n]$, and since $A\subset B$, $B$ must contain at least $n+f(n)$ elements from $[2n]$ as well. This however contradicts the definition of $\F_m$. 

We show that if we pick $f(n)= \Theta(\sqrt{n\log\log n})$, then $\F_n$ has the desired size. The main difficulty of our proof comes from estimating the size of $\F_n$.

\begin{proof}[Proof of Theorem~\ref{thm:constr}]
Let $f(x):=\lfloor 3\sqrt{x\log\log (x+3)} \rfloor + 100$ and
\begin{equation}
        \F_n=\{A\subset [2n]: |A|=n+f(n), ~ \forall i: 0<i<n\implies  |A\cap [2i]|<i+f(i) \}.
\end{equation}
Set $\F:=\bigcup_{n=1}^{\infty} \F_n$. Then $\F$ is an antichain, so it only remains to show that $$\liminf_{n\rightarrow \infty}|\F_n|\cdot\frac{n\log^{46} n}{2^n}>0.$$

 It will be helpful for us to identify sets with lattice paths that start at $(0,0)$ and take steps in directions $(1,1)$ and $(1,-1)$. By identifying the terms ``contains element $i$'' and ``the $i$-th step is in direction $(1,1)$'', we find a bijection between $\F_n$ and the family $\G_n$, defined as the collection of paths starting at $(0,0)$, ending at $(2n,2f(n))$, taking steps in directions $(1,1)$ and $(1,-1)$, and staying strictly below the curve $y=2f(x/2)$ for all even $x$ satisfying $0<x<2n$. See Figure \ref{fig:my_label} for an illustration.
 
 \begin{figure}
     \centering
    \includegraphics[scale=0.75]{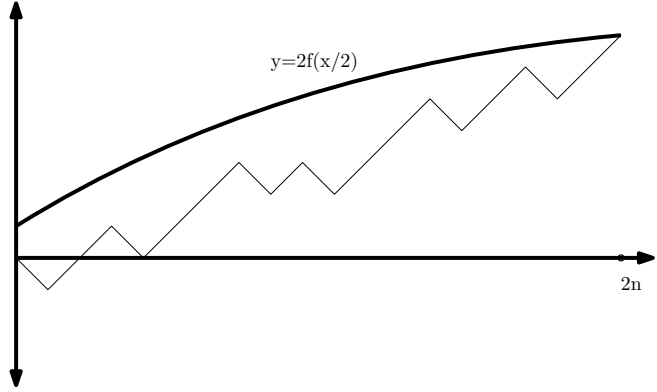}
     \caption{An element of the family $\G_n$}
     \label{fig:my_label}
 \end{figure}
 
 The following statement is a version of the Law of the Iterated Logarithm (see e.g.~\cite{probbook}). For the sake of the reader we prove here a specific form which we need for our construction.
\begin{lemma}\label{claim:lil}
 Let $X_1,X_2,\ldots$ be independent random variables with $\mathbb{P}(X_i=1)=\mathbb{P}(X_i=-1)=1/2$ for each $i$. For all $n$, let $S_n:= X_1+\ldots +X_n$. Then the probability that there exists some $n$ with $S_n>3\sqrt{n\log\log (n+3)}+100$ is at most $1/2$.
 \end{lemma}
 \begin{proof}
 Let $C>0$ be a positive constant. For all $i$, let $A_i$ be the event that $i$ is the smallest integer with $S_i>C\sqrt{i\log\log (i+3)}+100$, and observe that for $i\leq 100$ the event $A_i$ is empty. Fix an arbitrary  integer $n$, and in the first step of the proof we will show an upper bound for the probability that $A_i$ occurs for some $i$ with $\lceil n/2 \rceil \leq i \leq n$. Let $U_j$ be the event that the number of 1's amongst $X_{j+1},X_{j+2},\ldots,X_n$ is at least $(n-j)/2$, thus clearly $\mathbb{P}(U_j)\geq 1/2$ for all $j$. From definitions, one can easily check that for all $j\geq n/2\geq 100$, we have 
 \begin{equation}\label{eq:lil1}
     A_j\cap U_j \implies \left\{S_n > \frac{2}{3}C\sqrt{n\log\log n}\right\}.
 \end{equation}
Therefore,
 \begin{equation}\label{eq:lil2}
     \mathbb{P}\left(\bigcup_{j=\lceil n/2\rceil}^n A_j\cap U_j\right) \leq  \mathbb{P}\left(S_n > \frac{2}{3}C\sqrt{n\log\log n}\right).
 \end{equation}
 As $A_j$ and $U_j$ are independent, we can write
 \begin{equation} \label{eq:lil3} 
     \mathbb{P}(A_j\cap U_j)=\mathbb{P}(A_j)\mathbb{P}(U_j)\geq \frac{\mathbb{P}(A_j)}{2}. 
 \end{equation}
 Putting equations~(\ref{eq:lil2}) and~(\ref{eq:lil3}) together, and using that the $A_j$-s, and hence the $A_j\cap U_j$-s, are disjoint events, we get
 \begin{equation}
 \begin{split}
     \mathbb{P}\left(\bigcup_{j=\lceil n/2\rceil}^n A_j\right) &= \sum_{j=\lceil n/2\rceil}^n \mathbb{P}(A_j) \leq 2\sum_{j=\lceil n/2\rceil}^n \mathbb{P}(A_j\cap U_j) = \\
     &= 2\cdot \mathbb{P}\left(\bigcup_{j=\lceil n/2\rceil}^n A_j\cap U_j\right) \leq 2\cdot \mathbb{P}\left(S_n > \frac{2}{3}C\sqrt{n\log\log n}\right).
     \end{split}
 \end{equation}
 We can bound the last quantity using the standard Chernoff bound, to get
 $$ \mathbb{P}\left(\bigcup_{j=\lceil n/2\rceil}^n A_j\right)\leq \mathbb{P}\left(S_n > \frac{2}{3}C\sqrt{n\log\log n}\right) \leq e^{-\frac{2}{9}C^2\log\log n}.$$
 Now we are ready to bound the probability that an $A_i$ occurs for some $i\leq N$, where $N$ is an arbitrary large integer of the form $N= 100\cdot 2^k$.
 \begin{equation}
 \begin{split}
     \mathbb{P}\left(\bigcup_{j=100}^N A_j\right)&\leq   \mathbb{P}\left(\bigcup_{j=N/2}^{N} A_j\right) + \mathbb{P}\left(\bigcup_{j=N/4}^{N/2} A_j\right) + \ldots + \mathbb{P}\left(\bigcup_{j=100}^{200} A_j\right)  \\&\leq  e^{-\frac{2}{9}C^2\log\log N} + e^{-\frac{2}{9}C^2\log\log (N/2)} + \ldots + e^{-\frac{2}{9}C^2\log\log 200} \\&=
     \sum_{i=1}^{k} (\log100 + i\log2)^{-\frac{2}{9}C^2}
     \end{split}
 \end{equation}
 Setting $C=3$ then gives $\mathbb{P}\left(\bigcup_{j=100}^N A_j\right)\leq\frac{1}{2}$. As this bound holds uniformly for all $N$, the claim follows.
 \end{proof}
 
 For $k\in\mathbb{Z}$, denote by $P(k)$ the collection of paths starting at $(0,0)$, ending at $(n,k)$, taking steps in directions $(1,1)$ and $(1,-1)$ (such paths exist precisely if $n+k$ is even), and staying strictly below the curve $y=2f(x/2)$ for all even $x$ satisfying $0<x\leq n$.
 Similarly, let $Q(k)$ denote the collection of paths starting at $(n,k)$, ending at $(2n,2f(n))$, taking steps in directions $(1,1)$ and $(1,-1)$, and staying strictly below the curve $y=2f(x/2)$ for all even $x$ satisfying $x< 2n$. Observe that 
 \begin{equation}\label{eq:gbound}
 |\G_n|=\sum_k |P(k)||Q(k)|,
 \end{equation} hence, it is enough to find suitable lower bounds for $|P(k)|$ and $|Q(k)|$. Let $$A:=2\left(\frac{5}{2}\sqrt{\frac{n}{2}\log\log \frac{n+3}{2}} + 90\right).$$
 
 \begin{claim}\label{lemma:Pk}
$$\sum_{k=-A}^{A} |P(k)|\geq 2^{n-1}.$$
 \end{claim}
 \begin{proof}
 It is a straightforward calculation to verify that  $$2f(n/2)\geq A\geq 3\sqrt{n\log\log (n+3)} + 100.$$
 Hence, by Lemma~\ref{claim:lil},  we have 
$\sum_{k=-A}^{A} |P(k)|\geq 2^{n-1}.$
 \end{proof}

The more difficult part is to derive a lower bound for $|Q(k)|$.
\begin{claim}\label{lemma:Qk}
If $-A\leq k\leq A$, then 
$$|Q(k)|\geq \frac{2^n}{n\log^{46}n}.$$
\end{claim}
\begin{proof}
Say that a path is \emph{good} if it starts at $(0,0)$, ends at $(n,2f(n)-k)$, and with the exception of the second endpoint, stays strictly below the straight line connecting $(0,2f(n/2)-k)$ and $(n,2f(n)-k)$. This line is given by the formula $y=\mu x + c$ with $\mu = \frac{2f(n)-2f(n/2)}{n}$ and $c = 2f(n/2)-k$. Then the number of good paths is a lower bound for $|Q(k)|$, as every good path shifted by the vector $(n,k)$ gives rise to a path in $Q(k)$.

 Note that the number of good paths is equal to the number of paths that start at $(0,0)$, end at  $(n,2f(n)-k)$, and with the exception of the first endpoints, stay strictly above the line $y=\mu x$. Indeed, given a good path with steps $p_1,\ldots, p_n$ where $p_i\in\{(1,1),(1,-1)\}$ for all $i$, the reverse path $p_n,\ldots,p_1$ is of the second type, and vice versa.  We will count the number of these paths as follows. Take any path from the origin to $(n,2f(n)-k)$ using steps $(1,1)$ and $(1,-1)$. Such a path contains $p=\frac{1}{2}\left(n+2f(n)-k\right)$ steps in the $(1,1)$ direction and $q=\frac{1}{2}\left(n-2f(n)+k\right)$ steps in the $(1,-1)$ direction. This path can be represented by a sequence $x_1,\ldots,x_n$ of +1's and -1's, where $x_i=1$ precisely if the $i$'th step was in the $(1,1)$ direction and $x_i=-1$ otherwise. Say that the sequence $x_1,\ldots,x_n$, where the number of $+1$'s is $p$ and the number of $-1$'s is $q$, is \emph{good} if $\sum_{i=1}^j x_i > \mu j$ for all $1\leq j\leq n$. Then the number of good sequences is equal to the number of good paths.
 
% The restriction that the path stays above $y=\mu x$ can be written as $\sum_{i=1}^j x_i > \mu j$ for all $1\leq j\leq n$. This is equivalent to saying that on every non-empty initial sequence the number of $+1$'s is greater than $\alpha$ times the number of $-1$'s, where  $\alpha = \frac{1+\mu}{1-\mu}$.
 
Given a sequence $x_1,\ldots,x_n$ of $+1$'s and $-1$'s, arrange them in clockwise order around a circle, looping around only once, so that $x_1$ comes after $x_n$ in the clockwise direction. Say that two $\{-1,+1\}$ sequences are equivalent if we can get one from the other by clockwise rotation, that is, $x_1,\dots,x_n$ is equivalent to $y_1,\dots,y_n$ if there exists $r\in [n]$ such that $x_i=y_{i+r}$ for every $i\in [n]$, where indices are meant modulo $n$.  We show that each equivalence class contains at least $p-\lfloor \frac{1+\mu}{1-\mu} q\rfloor$ good sequences. In order to prove this, we need a variant of the well known ``cycle lemma'' (see e.g.~\cite{cycle1}). Although this lemma appeared in many forms in different papers, the variant which we need is not so common and can be found in~\cite{cyclelemma} (corollary of Theorem 3). 
 
Given a sequence $x_1,\ldots,x_n$ of $+1$'s and $-1$'s in circular order and the number $\mu$, say that an index $r$ is a \emph{head} if $\sum_{i=r}^{r+j-1}x_{i}> \mu j$ for all $j\in [n]$. Equivalently, the number of $+1$'s on every nonempty arc beginning with $x_r$ and proceeding in the positive direction is greater than $\frac{1+\mu}{1-\mu}$ times the number of $-1$'s on the same arc. Clearly, if there are $p$ $+1$'s and $q$ $-1$'s in the sequence, then each head is the starting point of a good sequence in the circular order. Here we give a different proof from that in~\cite{cyclelemma} which has an additional benefit of explicitly pointing to the locations of heads.

\begin{lemma}\label{lem:cycle}
Let $x_1,x_2,\ldots,x_{p+q}$ be a sequence with values in $\{-1,+1\}$ such that the number of $+1$'s is $p$ and the number of $-1$'s is $q\neq 0$. Arrange this sequence around a circle as above, and let $\mu\in\mathbb{R}$ satisfy $0<\mu \leq \frac{p-q}{p+q}$. Then this sequence has at least $p-\lfloor \frac{1+\mu}{1-\mu} q\rfloor$ distinct heads. 
\end{lemma} 
\begin{proof}
For all $i$, let $S_i=(x_1+\ldots + x_i) - \mu i$. Let $t:=\min_{0\leq i \leq n}S_i$, and note that since we set $S_0=0$ we have $t\leq 0$. Observe that $S_n = p-q-\mu (p+q)=(1-\mu)(p- \frac{1+\mu}{1-\mu} q)$. For $1\leq i \leq p-\lfloor  \frac{1+\mu}{1-\mu} q \rfloor$, let $\gamma_i$ be the largest index $\gamma$ such that $S_\gamma < t + i(1-\mu)=:d_i$. Note that for $i$ in this range, $d_i\leq S_n$ ; moreover the $\gamma_i$-s are well-defined and distinct because $S_{j+1}-S_j\leq 1-\mu$ for all $j$. We show that $\gamma_i+1$ is a head for $1\leq i \leq p-\lfloor  \frac{1+\mu}{1-\mu} q \rfloor$.

Fix an $i$ with $0\leq i \leq p-\lfloor \frac{1+\mu}{1-\mu} q \rfloor$ and for $j\in [n]$, let $T_j = (x_{\gamma_i+1} + x_{\gamma_i+1}+\ldots + x_{\gamma_i+j})-\mu j$, where the subscripts are interpreted modulo $n$. We claim that $T_j$ is strictly positive. If $j$ is such that $\gamma_i+j\leq n$ then $T_j>0$ because $\gamma_i$ was the largest index with $S_i< d_i$. If $j$ is larger, then we get
\begin{equation*}
    \begin{split}
        T_j &= (x_{\gamma_i+1}+\ldots+x_n) - \mu(n-\gamma_i) + (x_1+\ldots + x_{j-n+\gamma_i}) - \mu(j-n+\gamma_i)
        \\&=S_n-S_{\gamma_i} + S_{j-n+\gamma_i}\geq (1-\mu)\Big(p- \frac{1+\mu}{1-\mu}q\Big)-d_i + t \\&= (1-\mu) \Big(p- \frac{1+\mu}{1-\mu}q\Big) - (d_i-t) = \Big(p- \frac{1+\mu}{1-\mu}q -i\Big)(1-\mu)\geq 0.
    \end{split}
\end{equation*}
Thus $T_j>0$ for all $j\in [n]$, and hence $\gamma_i+1$ is a head as claimed.
\end{proof}

For the rest of the proof we will assume that $n$ is large. Recall that 
$$
     p= \frac{1}{2}\left(n+2f(n)-k\right), \quad \quad  q = \frac{1}{2}\left(n-2f(n)+k\right), \quad \quad \mu = \frac{2f(n)-2f(n/2)}{n},$$
$$     A=2\left(\frac{5}{2}\sqrt{\frac{n}{2}\log\log \frac{n+3}{2}} + 90\right) \quad \mbox{and} \quad
     f\left(\frac{n}{2}\right) = \Bigg\lfloor 3\sqrt{\frac{n}{2}\log\log\left(\frac{n}{2}+3\right)}\Bigg\rfloor + 100.$$

Since $ k \leq A \leq 2f(n/2)$,  we have
$$\frac{1+\mu}{1-\mu} = \frac{n+2f(n)-2f\left(\frac{n}{2}\right)}{n-2f(n)+2f\left(\frac{n}{2}\right)}\leq \frac{n+2f(n)-2f\left(\frac{n}{2}\right)}{2q},$$ and so 
$$p-\frac{1-\mu}{1+\mu}q \geq f\left(\frac{n}{2}\right)-\frac{k}{2}\geq \frac{1}{4}\sqrt{n\log\log n}.$$
In particular, $p-\frac{1-\mu}{1+\mu}q >0$ implies that $\mu \leq \frac{p-q}{p+q}$. Thus we can apply Lemma~\ref{lem:cycle}  to conclude that there are always at least $p-\lfloor \frac{1+\mu}{1-\mu}\ q\rfloor\geq \frac{1}{4}\sqrt{n\log\log n}$ heads. This shows that for large $n$, any path from the origin to $(n,2f(n)-k)$ has at least $\frac{1}{4}\sqrt{n\log\log n}$ cyclic shifts which stay above the line $y=\mu x$.

Observe that this does not immediately imply that there are $\frac{1}{4}\sqrt{n\log\log n}$ good sequences in any equivalence class, as two heads may correspond to the same good sequence. Fix a good sequence, let  $E$ be the collection of all distinct sequences obtained from it by all cyclic shifts. Since cyclic shift is a group action and there are $n$ such shifts, this means that each sequence in our collection is obtained by shifts exactly $n/|E|$ times. Therefore there are at least 
$\frac{\frac{1}{4}\sqrt{n\log\log n}}{n/|E|}=\frac{\frac{1}{4}\sqrt{n\log\log n}}{n}|E|$ good sequences in $E$.  Since this is true for any equivalence class, the total fraction of good sequences is at least $\frac{\frac{1}{4}\sqrt{n\log\log n}}{n}$. 

Hence, we obtain
$$|Q(k)|\geq \frac{\frac{1}{4}\sqrt{n\log\log n}}{n}\binom{n}{\frac{n}{2} + f(n)-\frac{k}{2}},$$
whenever $n+k$ is even. Using that $-A\leq k\leq A$ we can write 
$$|Q(k)|\geq \frac{\sqrt{n\log\log n}}{4n}\binom{n}{\frac{n}{2} + f(n)+\frac{A}{2}}\geq \frac{1}{\sqrt{n}} \frac{2^{n}}{\sqrt{n}} e^{-2(f(n)+A/2)^2/n}\geq \frac{2^n}{n\log^{46}n}, $$ 
where for the second inequality we used the known estimate for binomial coefficients, see \cite{asymptopia}, equation~(5.41). This finishes the proof of Claim~\ref{lemma:Qk}.
\end{proof}

By  Claim \ref{lemma:Pk} and Claim \ref{lemma:Qk}, we have that for sufficiently large $n$,
\begin{equation*}
    |\F_n| = \sum_{k}|P(k)||Q(k)|\geq \left(\sum_{k=-A}^{k=A} |P(k)|\right)\cdot\min_{\substack{-A\leq k\leq A\\n+k\text{ is even}}} |Q(k)|\geq \frac{2^{2n}}{2n\log^{46} n}.
\end{equation*}
This finishes the proof of Theorem~\ref{thm:constr}.
\end{proof}

\section{Concluding remarks and open problems}\label{sec:concl}

In this note we study how certain set-theoretic extremal results behave in a finite versus infinite setting, by analyzing this problem for antichains. As already mentioned, being an antichain is a much stronger requirement than being prefix-free, yet our construction gives a similar estimate as the best construction for Kraft's inequality. We further speculate that the upper bound coming from Kraft's inequality might be essentially correct, that is, for any $\epsilon >0$ there exists an antichain $\F\subset 2^{\mathbb{N}}$ such that $$\liminf_{n\rightarrow \infty}\frac{|\F \cap 2^{[n]}|\cdot n\log^{1+\epsilon} n}{2^n}>0.$$ 

Another interesting open problem is to study infinite $H$-free graphs (for bipartite $H$) and the growth rate of the minimum degree of their restrictions to $[n]$. As we already mentioned in the introduction, the minimum degree of $C_4$-free graphs and more generally for $K_{s,t}$-free graphs was studied by Conlon--Fox--Sudakov~\cite{benny!}. They showed that if $G$ is a $K_{s,t}$-free graph on $\mathbb{N}$ where $2\leq s\leq t$, then 
$$\liminf_{n\rightarrow\infty} \frac{\delta(G_n)(\log n)^{1/s}}{n^{1-1/s}}<\infty.$$ A natural question, raised in \cite{benny!} whether it is true for any infinite $C_6$-free graph $G$ that
$$\liminf_{n\rightarrow\infty} \frac{\delta(G_n)}{n^{1/3}}=0.$$ 

Finally, it would be interesting to obtain further results in this spirit of ``finite versus infinite'' for some other problems in extremal combinatorics.

\vspace{0.3cm}
\noindent
{\bf Acknowledgment.}\, We would like to thank Vincent Tassion and Wendelin Werner for useful discussions about the Law of Iterated Logarithm, and D\"om\"ot\"or P\'alv\"olgyi for pointing out that our proof of Theorem \ref{thm:upper} uses Kraft's inequality.

\bibliography{mybib}
\bibliographystyle{abbrv}

\end{document}